\documentclass[11pt,leqno]{amsart}
\usepackage{amsmath,amssymb,amsthm}
\usepackage[colorlinks=false, pdfborder={0 0 0}]{hyperref}

\usepackage{amsfonts,amsmath, amsthm, amscd, amssymb, graphicx, color, mathrsfs, stmaryrd}
\usepackage[utf8]{inputenc}
\usepackage{tikz-cd}
\usepackage{adjustbox}
\usepackage{mathtools}
\usepackage{enumerate}

\usepackage{ulem}

\usepackage{bbm}

\newcommand{\R}{\mathbb{R}}
\newcommand{\T}{\mathbb{T}}

\renewcommand{\geq}{\geqslant}
\renewcommand{\leq}{\leqslant}

\renewcommand{\>}{\rangle}

\newcommand{\norm}[1]{\left\Vert#1\right\Vert}

\newcommand{\lip}{{\mathrm{lip}}_0}

\newcommand{\Lip}{{\mathrm{Lip}}_0}

\newcommand{\NA}{\operatorname{NA}}

\newcommand{\co}{\operatorname{co}}

\newcommand{\cconv}{\overline{\co}}

\newtheorem{thm}{Theorem}[section]

\newtheorem{prop}[thm]{Proposition}
\newtheorem{coro}[thm]{Corollary}
\newtheorem{lema}[thm]{Lemma}

\theoremstyle{definition}

\newtheorem{ejem}[thm]{Example}
\newtheorem{rema}[thm]{Remark}

\numberwithin{equation}{section}

\def\fnote#1{\footnote}

\def\natu{{\mathbb N}}

\def\ignora#1{}

\def\n3#1{\left\vert  \! \left\vert \! \left\vert \, #1 \, \right\vert \!
  \right\vert \! \right\vert }

\newcommand{\iten}{\ensuremath{\widehat{\otimes}_\varepsilon}}
\newcommand{\pten}{\ensuremath{\widehat{\otimes}_\pi}}

\newcommand{\ptensor}{\widehat{\otimes}_{\pi}}
\newcommand{\lipC}{{\mathrm{lip}}_\tau}

\newcommand{\renorm}[1]{\left\lvert\!\left\lvert\!\left\lvert #1 \right\rvert\!\right\rvert\!\right\rvert}

\begin{document}

\title[]{Projective tensor products where every element is norm-attaining}

\author[L. Garc\'ia-Lirola]{Luis C. Garc\'ia-Lirola }\address{Departamento de Matemáticas, Universidad de Zaragoza, 50009, Zaragoza, Spain} 
\email{\texttt{luiscarlos@unizar.es}}

\author{ Juan Guerrero-Viu }
\address{Departamento de Matemáticas, Universidad de Zaragoza, 50009, Zaragoza, Spain} 
\email{815649@unizar.es}

\author{ Abraham Rueda Zoca }\address{Universidad de Granada, Facultad de Ciencias. Departamento de An\'{a}lisis Matem\'{a}tico, 18071-Granada
(Spain)} \email{ abrahamrueda@ugr.es}
\urladdr{\url{https://arzenglish.wordpress.com}}

\keywords{Banach spaces; Projective tensor product; projective norm-attainment; norm-attaining operators.}

\subjclass{46B04, 46B20, 46B28}

\begin{abstract}
In this paper we analyse when every element of $X\pten Y$ attains its projective norm. We prove that this is the case if $X$ is the dual of a subspace of a predual of an $\ell_1(I)$ space and $Y$ is $1$-complemented in its bidual under approximation properties assumptions. This result allows us to provide some new examples where $X$ is a Lipschitz-free space.  We also prove that the set of norm-attaining elements is dense in $X\pten Y$ if, for instance, $X=L_1(\mu)$ and $Y$ is any Banach space, or if $X$ has the metric $\pi$-property and $Y$ is a dual space with the RNP.  
\end{abstract}

\maketitle

\markboth{Luis C. Garc\'ia-Lirola, Juan Guerrero-Viu and Abraham Rueda Zoca}{Projective tensor products where every element is norm-attaining}

\section{Introduction}

The study of norm-attaining functions has produced a vast literature in Banach space theory during the last 60 years in many different contexts (e.g. for bounded linear operators \cite{bourgain,lindenstrauss,martin16}, for bounded bilinear mappings \cite{acagpa,Choi} or for Lipschitz functions spaces \cite{CGMR,Godefroy, kms}). 

Very recently, motivated by finding a suitable notion of norm-attainment on nuclear operators, a notion of norm-attainment for elements in a projective tensor product of two Banach spaces was introduced in \cite{DJRRZ} (see Subsection~\ref{subsec:normattaintensor} for background). Namely, given $z\in X\pten Y$, it is said that $z$ \textit{attains its projective norm} (that is, $z\in \NA_\pi(X, Y)$) if there exist sequences $(x_n)_n\subset X$, and $(y_n)_n\subset Y$ such that
\[z=\sum_{n=1}^\infty x_n\otimes y_n,\quad \text{and} \quad \norm{z}_{\pi}=\sum_{n=1}^\infty \norm{x_n}\norm{y_n}.\]
After that, new results have appeared in the literature in \cite{DGLJRZ,rueda23}. We collect here the main results in this line:
\begin{enumerate}
    \item $\NA_\pi(X\pten Y)=X\pten Y$ if $X$ and $Y$ are finite-dimensional \cite[Proposition 3.5]{DJRRZ}, $X=\ell_1(I)$ for some set $I$ and $Y$ is any space \cite[Proposition 3.6]{DJRRZ}, if $X=Y$ is a complex Hilbert space 
    (based on a diagonalization argument) \cite[Proposition 3.8]{DJRRZ} and if $X$ is a finite-dimensional polyhedral Banach space and $Y$ is any dual Banach space \cite[Theorem 4.1]{DGLJRZ}.
    \item $\NA_\pi(X\pten Y)$ is dense in $X\pten Y$ if $X$ has the metric $\pi$-property and $Y$ either is uniformly convex or it has the metric $\pi$-property \cite[Theorem 4.8]{DJRRZ}, if $X$ is a polyhedral Banach space with the metric $\pi$-property and $Y$ is a dual Banach space \cite[Theorem 4.2]{DGLJRZ} or if $X$ and $Y$ are dual spaces with the RNP and either $X^*$ or $Y^*$ has the approximation property \cite[Corollary 4.6]{DGLJRZ}.
    \item There are Banach spaces $X$ and $Y$ such that $\NA_\pi(X\pten Y)$ is not dense in $X\pten Y$ \cite[Theorem 5.1]{DJRRZ} (in particular, in such space there are non norm-attaining tensors which are expressible as finite sum of basic tensors).
    \item $X\pten Y\setminus \NA_\pi(X\pten Y)$ is dense if $X=c_0(I)$ and $Y$ is a Hilbert space \cite[Theorem 3.1]{rueda23}. In particular, it is possible that both $\NA_\pi(X\pten Y)$ and its complement are simultaneously dense in $X\pten Y$.
\end{enumerate}

After this list, it is clear that most of the existing results concern the density of $\NA_\pi(X\pten Y)$, but the phenomenon of when $\NA_\pi(X\pten Y)=X\pten Y$ (in other words, when every element of $X\pten Y$ attains its projective norm) is quite unknown. The main aim of the results of Section~\ref{section:todoelealcanza} is to shed light on this phenomenon. As the main result of this section, we prove in Theorem~\ref{teo:subspredl1attains} that if $X$ is a subspace of an $\ell_1$ predual (i.e. $X\subseteq Z$ for a Banach space $Z$ such that $X^*=\ell_1(I)$ holds isometrically) and $Y$ is dual Banach space, then $\NA_\pi(X^*\pten Y)=X^*\pten Y$ if either $X^*$ or $Y$ has the approximation property.

In the search of necessary conditions, we prove in Corollary~\ref{cor:propAnormatensor} that if $X$ is a separable Banach space such that $\NA_\pi(X^*\pten Y^*)=X^*\pten Y^*$ holds for every Banach space $Y$ then $B_{X^*}=\overline{\co}(\exp(B_{X^*}))$. For proving that we adapt a necessary condition for Lindentrauss property A exhibited in \cite[Theorem 2]{lindenstrauss}. Moreover, we prove in Theorem~\ref{theo:complebig} a criterion in order to get that the set of the elements that do not attain their projective norm is norming for the dual (and in particular this set is big), showing explicit examples where this criterion applies in Example~\ref{exam:fovelle}. In particular, we obtain examples of Banach spaces $X$ and $Y$ such that $\NA_\pi(X\pten Y)$ is dense in $X\pten Y$ and $B_{X\pten Y}=\overline{\co}(\{z\in S_{X\pten Y}: z\notin \NA_\pi(X\pten Y)\})$ (see Remark~\ref{remark:bignaandcomple}).

In Section~\ref{sect:densidad} we obtain new results on density of $\NA_\pi(X\pten Y)$. We prove in Corollary~\ref{cor:denseness} that $\NA_\pi(X\pten Y)$ is dense in $X\pten Y$ if $X=L_1(\mu)$ and $Y$ is any Banach space, if $X$ is an $L_1$-predual and $Y$ is $1$-complemented in $Y^{**}$ and if $X$ has the metric $\pi$-property and $Y$ is an RNP dual Banach space. 

\begin{rema} In \cite{DJRRZ}, the concept of \textit{nuclear operator which attains its nuclear norm} is considered, and a detailed treatment is done (see \cite[Section 2.3]{DJRRZ} for the (well-known) connection between nuclear operators and projective tensor products under approximation properties assumptions). All the results in the paper have a translation in the context of nuclear operators. However, we are not going to write them explicitly  
to avoid extra notation, duplicity of results, and  
technicalities which may make harder to understand the ideas behind the proofs.
\end{rema}

\section{Notation and preliminary results} 

For simplicity we will consider real Banach spaces. We denote by $B_X$ and $S_X$ the closed unit ball and the unit sphere, respectively, of the Banach space $X$. We denote by $\mathcal L(X, Y)$ the space of all bounded linear operators from $X$ into $Y$. If $Y = \mathbb R$, then $\mathcal L(X, \mathbb R)$ is denoted by $X^*$, the topological dual space of $X$. We denote $\mathcal F(X,Y)$ the space of finite-rank operators from $X$ into $Y$.

\subsection{Tensor product spaces}

The projective tensor product of $X$ and $Y$, denoted by $X \pten Y$, is the completion of the algebraic tensor product $X \otimes Y$ endowed with the norm
$$
\|z\|_{\pi} := \inf \left\{ \sum_{n=1}^k \|x_n\| \|y_n\|: z = \sum_{n=1}^k x_n \otimes y_n \right\},$$
where the infimum is taken over all such representations of $z$. The reason for taking completion is that $X\otimes Y$ endowed with the projective norm is complete if, and only if, either $X$ or $Y$ is finite dimensional (see \cite[P.43, Exercises 2.4 and 2.5]{ryan}).

It is well known that $\|x \otimes y\|_{\pi} = \|x\| \|y\|$ for every $x \in X$, $y \in Y$, and that the closed unit ball of $X \pten Y$ is the closed convex hull of the set $B_X \otimes B_Y = \{ x \otimes y: x \in B_X, y \in B_Y \}$. Throughout the paper, we will use of both facts without any explicit reference.

Observe that the action of an operator $G\colon X \longrightarrow Y^*$ as a linear functional on $X \pten Y$ is given by
$$
G \left( \sum_{n=1}^{k} x_n \otimes y_n \right) = \sum_{n=1}^{k} G(x_n)(y_n),$$
for every $\sum_{n=1}^{k} x_n \otimes y_n \in X \otimes Y$. This action establishes a linear isometry from $\mathcal L(X,Y^*)$ onto $(X\pten Y)^*$ (see e.g. \cite[Theorem 2.9]{ryan}). All along this paper we will use the isometric identification $(X\pten Y)^*= \mathcal L(X,Y^*)$ without any explicit mention.

Furthermore, given two bounded operators $T\colon X\longrightarrow Z$ and $S\colon Y\longrightarrow W$, we can define an operator $T\otimes S\colon X\pten Y\longrightarrow Z\pten W$ by the action $(T\otimes S)(x\otimes y):=T(x)\otimes S(y)$ for $x\in X$ and $y\in Y$. It follows that $\Vert T\otimes S\Vert=\Vert T\Vert\Vert S\Vert$. Moreover, it is known that if $T,S$ are bounded projections then so is $T\otimes S$. Consequently, if $Z\subseteq X$ is a $1$-complemented subspace, then $Z\pten Y$ is a $1$-complemented subspace of $X\pten Y$ in the natural way (see \cite[Proposition 2.4]{ryan} for details).

Recall that a Banach space $X$ has the \textit{metric approximation property} (MAP)
if there exists a net $(S_\alpha)_\alpha$ in $\mathcal F(X,X)$ with $\Vert S_\alpha\Vert\leq 1$ for every $\alpha$ and such that
$S_\alpha (x) \to x$ for all $x \in X$. In the above condition, if $S_\alpha$ is additionally a projection, we say that $X$ has the \textit{metric $\pi$-property}. Equivalently, $X$ has the metric $\pi$-property if given $\varepsilon>0$ and $\{x_1,\ldots, x_n\}\subseteq S_X$, we can find a finite dimensional 1-complemented subspace $M\subset X$ such that for each $i\in\{1,\ldots, n\}$ there exists $x_i'\in M$ with $\norm{x_i-x_i'}<\varepsilon$ (see \cite{DJRRZ}). 

Recall that given two Banach spaces $X$ and $Y$, the
\textit{injective tensor product} of $X$ and $Y$, denoted by
$X \iten Y$, is the completion of $X\otimes Y$ under the norm given by
\begin{equation*}
   \Vert u\Vert_{\varepsilon}:=\sup
   \left\{
      \sum_{i=1}^n \vert x^*(x_i)y^*(y_i)\vert
      : x^*\in S_{X^*}, y^*\in S_{Y^*}
   \right\},
\end{equation*}
where $u=\sum_{i=1}^n x_i\otimes y_i$ (see \cite[Chapter 3]{ryan} for background).
Every $u \in X \iten Y$ can be viewed as an operator $T_u : X^* \rightarrow Y$ which is weak$^*$-to-weakly continuous. Under this point of view, the norm on the injective tensor product is nothing but the operator norm.

As well as happen with the projective tensor product, given two bounded operators $T\colon X\longrightarrow Z$ and $S\colon Y\longrightarrow W$, we can define an operator $T\otimes S\colon X\iten Y\longrightarrow Z\iten W$ by the action $(T\otimes S)(x\otimes y):=T(x)\otimes S(y)$ for $x\in X$ and $y\in Y$. It follows that $\Vert T\otimes S\Vert=\Vert T\Vert\Vert S\Vert$. However, this time we get that if $T,S$ are linear isometries then $T\otimes S$ is also a linear isometry (c.f. e.g. \cite[Section 3.2]{ryan}). This fact is commonly known as ``the injective tensor product respects subspaces''.

It is known that, given two Banach spaces $X$ and $Y$, we have $(X\iten Y)^*=X^*\pten Y^*$ if either $X^*$ or $Y^*$ has the Radon-Nikodym Property (RNP) and either $X^*$ or $Y^*$ has the approximation property (AP) \cite[Theorem 5.33]{ryan}.

\subsection{Projective norm attainment}\label{subsec:normattaintensor}

One consequence of the isometric identification $\ell_1(I)\pten X=\ell_1(I,X)$ is \cite[Proposition 2.8]{ryan}, which establishes that, given two Banach spaces $X$ and $Y$, then for every $z\in X\pten Y$ and every $\varepsilon>0$, there exist sequences $(x_n)_n$ in $X$ and $(y_n)_n$ in $Y$ with $u=\sum_{n=1}^\infty x_n\otimes y_n$ (where the above convergence is in the norm topology of $X\pten Y$) and such that $\Vert z\Vert_{\pi}\leq \sum_{n=1}^\infty \Vert x_n\Vert\Vert y_n\Vert\leq \Vert z\Vert_{\pi}+\varepsilon$. Consequently, it follows that
$$\Vert z\Vert_{\pi}=\inf\left\{\sum_{n=1}^\infty \Vert x_n\Vert\Vert y_n\Vert: \sum_{n=1}^\infty \Vert x_n\Vert\Vert y_n\Vert<\infty,  u=\sum_{n=1}^\infty x_n\otimes y_n \right\}$$
where the infimum is taken over all of the possible representations of $z$ as limit of a series in the above form.

According to \cite[Definition 2.1]{DJRRZ}, an element $z\in X\pten Y$ is said to \textit{attain its projective norm} if the above infimum is actually a minimum, that is, if there exists a sequence $(x_n)_n$ in $X$ and $(y_n)_n$ in $Y$ such that $\Vert z\Vert=\sum_{n=1}^\infty \Vert x_n\Vert \Vert  y_n\Vert$ and $z=\sum_{n=1}^\infty x_n\otimes y_n$. We denote $\NA_\pi(X\pten Y)$ the set of those $z$ which attain their projective norm.
Observe that for all $z\in \NA_\pi(X\pten Y)$ we can write $z=\sum_{n=1}^\infty \lambda_n x_n\otimes y_n$, where $x_n\in S_X, y_n\in S_Y$ for every $n\in\mathbb N$ and $\lambda_n \in \mathbb R^+$ satisfy $\sum_{n=1}^\infty \lambda_n=\Vert z\Vert$. In such case, given any $T\in  \mathcal L(X,Y^*)=(X\pten Y)^*$ such that $T(z)=\Vert z\Vert$ and $\Vert T\Vert=1$, a convexity argument implies that $T(x_n)(y_n)=\norm{T(x_n)}=\norm{T}=1$ for every $n\in\mathbb N$. In particular, $T$ attains its norm (see \cite{lindenstrauss}). We denote by $\NA(X,Y)$ the set of those norm-attaining operators.

\subsection{Lipschitz functions and Lipschitz-free spaces}

Given a metric space $M$, $B(x,r)$ denotes the closed ball in $M$ centered at $x\in M$ with radius $r$. A metric space $M$ is said to be \textit{proper} if every closed ball of $M$ is compact.

We will denote by $\Lip(M,X)$ (or simply $\Lip(M)$ if $X=\R$) the space of all $X$-valued Lipschitz functions on $M$ which vanish at a designated origin $0\in M$. We will consider the norm in $\Lip(M,X)$ given by the best Lipschitz constant, denoted $\Vert \cdot\Vert_L$.
We denote $\delta$ the canonical isometric embedding of $M$ into $\Lip(M)^*$, 
which is given by $\<f,\delta(x)\> = f(x)$ for $x\in M$ and $f\in \Lip(M)$. The space $\mathcal F(M)$ is defined as $\mathcal F(M)=\overline{\operatorname{span}}\{\delta(x):x\in M\}$. It is well known that $\mathcal F(M)^*=\Lip(M)$ isometrically.

In Section~\ref{section:todoelealcanza} we will be interested in spaces $\mathcal F(M)$ which are indeed dual Banach spaces. We will collect here a pair of known results in the literature which will be of particular interest for our purposes. Let us introduce a bit more of notation.

A function $f\colon M\to \mathbb R$ is said to be \textit{locally flat} if $\lim_{x,y\to p}\frac{|f(x)-f(y)|}{d(x,y)} =0$ for every $p\in M$, and it is said to be \textit{flat at infinity} if $\lim_{r\to\infty} \norm{f|_{M\setminus B(0,r)}}=0$. The space of \textit{little Lipschitz functions} $\lip(M)$ is made up of those functions in $\Lip(M)$ which are both locally flat and flat at infinity. 

In case $M$ is compact, $\lip(M)$ coincides with the subspace of $\Lip(M)$ of those functions that are \textit{uniformly locally flat}, that is, 
\[ \lip(M)=\left\{f\colon M\to \mathbb R: \forall \varepsilon >0\, \exists \delta>0 : 0<d(x,y)\leq\delta \Rightarrow \frac{|f(x)-f(y)|}{d(x,y)}\leq \varepsilon \right\}.\] 
Also, given a topology $\tau$ in $M$, we denote
$$\lipC(M):=\{f\in\lip(M): f \mbox{ is }\tau\mbox{-continuous} \}.$$

Let us finally collect the promised results on the duality of Lipschitz-free spaces.
\begin{thm}\label{thm:dualityfree}
Let $M$ be a separable metric space.
\begin{enumerate}[a)]
    \item If $M$ is countable and proper then $\lip(M)^*=\mathcal F(M)$ isometrically \cite[Theorem 2.1]{Dalet05}, see also \cite[Theorem 3.2]{AGPP}.
    \item If $M$ is bounded and there exists a 
    compact topology $\tau$ on $M$  
    satisfying that, for every $\varepsilon>0$, there exists $f\in\lipC(M)$ with $\Vert f\Vert\leq 1$ and $f(x)-f(y)>d(x,y)-\varepsilon$, then $\lipC(M)^*=\mathcal F(M)$ isometrically \cite[Theorem~6.2]{Kalton04}.
\end{enumerate}
\end{thm}

The hypotheses in b) are satisfied in the case that $(M,d)$ is a uniformly discrete separable bounded metric space and there is a compact Hausdorff topology $\tau$ on $M$ such that $d$ is $\tau$-lower semi-continuous \cite[Corollary 3.9]{GLPPRZ}.

\section{Projective tensor products where every tensor attains its norm}\label{section:todoelealcanza}

In this section we will deepen in the study of the phenomenon of when $\NA_\pi(X\pten Y)=X\pten Y$ holds.

To begin with, the following easy observation allows us to get rid of the assumption of duality in many results due to the well-known fact that every dual Banach space is $1$-complemented in its bidual \cite[p. 221]{RojoYAmarillo}.

\begin{lema}\label{lema:complemented} Let $X,Y$ be Banach spaces such that $\NA(X\pten Y)=X\pten Y$. Assume $Z\subset Y$ is a $1$-complemented subspace.  Then, $\NA(X\pten Z)=X\pten Z$.
\end{lema}

\begin{proof}
Take $z\in X\ptensor Z $. Considering $z$ in $X\ptensor Y= \NA(X\ptensor Y)$, we can find sequences $(u_n)_n \subseteq X$ and $(v_n)_n\subseteq Y$ satisfying
\begin{equation*}
    z=\sum_{n=1}^{\infty} u_n\otimes v_n \quad \text{and}\quad \|z\|_{X\ptensor Y} = \sum_{n=1}^{\infty} \|u_n\| \|v_n\|.
\end{equation*}
Since there exists a projection $P\colon Y\rightarrow Z$ with $\|P\|=1$, we can consider the projection $Id\otimes P$ from $X\ptensor Y$ onto $X\ptensor Z$. Therefore,
\begin{equation*}
    z= (Id\otimes P)(z) = \sum_{n=1}^{\infty} u_n \otimes P(v_n) \in X\ptensor Z.
\end{equation*}
Furthermore, we have that
\begin{equation*}
    \sum_{n=1}^{\infty} \|u_n\|\|P(v_n)\| \leq  \sum_{n=1}^{\infty} \|u_n\|\|v_n\| = \|z\|_{X\ptensor Y} =\|z\|_{X\ptensor Z},
\end{equation*}
since $Z$ is 1-complemented in $Y$. Thus, we conclude that $\|z\|_{X\ptensor Z} =$\\ 
$ \sum_{n=1}^{\infty} \|u_n\|\|P(v_n)\|$ and $z\in \NA(X\ptensor Z)$. 
\end{proof}

As a consequence we get the following result, which improves \cite[Theorem~4.1]{DGLJRZ}.

\begin{coro}
Let $X$ be a finite dimensional Banach space whose unit ball is a polytope. If $Y$ is a Banach space which is $1$-complemented in $Y^{**}$ then $\NA_\pi(X\pten Y)=X\pten Y$.
\end{coro}

\begin{proof}
Observe that $\NA_\pi(X\pten Y^{**})=X\pten Y^{**}$ by \cite[Theorem 4.1]{DGLJRZ}. Since $Y$ is $1$-complemented in $Y^{**}$, Lemma~\ref{lema:complemented} applies.
\end{proof}

Next we aim to provide a general condition guaranteeing that every tensor attains its projective norm that extends at once \cite[Proposition 3.6]{DJRRZ} and \cite[Theorem 4.1]{DGLJRZ}.

In order to do so, we present here a result of independent interest. This result is a slight improvement of the well-known fact that the adjoint of a linear isometry between Banach spaces is a quotient operator. Moreover, this result is well known for specialists on proximinality on Banach spaces, but we prefer to include a proof for the sake of completeness.

\begin{lema}{\label{lema:operadorcociente}}
    Let $X,Y$ be normed spaces and $T\colon X\rightarrow Y$ a linear isometry. Then, the operator $T^*\colon Y^*\rightarrow X^*$ is a quotient operator. In fact, given $x^*\in X^*$, there exists $y^*\in Y^*$ such that $T^*(y^*)=x^*$ and $\|y^*\|=\|x^*\|$.
\end{lema}

\begin{proof}
    Fix $x^*\in X^*$. Let $\phi \colon T(X)\rightarrow \R$ be the functional defined by $\phi(T(x)) = x^*(x)$ for every $x\in X$. Since $T$ is a linear isometry, it is clear that $\phi$ is linear and $\|\phi\|=\|x^*\|$. By Hahn-Banach theorem, there exists $y^*\in Y^*$, a linear extension of $\phi$ such that 
    \begin{align*}
        \|y^*\|&=\|\phi\|=\|x^*\|.
    \end{align*}
    Therefore, we have that
    \begin{equation*}
        T^*(y^*)(x) = y^*(T(x)) = \phi(T(x)) = x^*(x), \quad \forall x\in X.
    \end{equation*}
    Thus, $T^*(y^*)=x^*$.\end{proof}

Now we present the following result, which is one of the main theorems in this paper, which will allow us to extend considerably the class of examples of Banach spaces $X$ and $Y$ such that $\NA_\pi(X\pten Y)=X\pten Y$.

\begin{thm}\label{teo:subspredl1attains} Let $Z$ be a Banach space such that $Z^*=\ell_1(I)$ isometrically. Let $X$ be a subspace of $Z$ and $Y$ be any Banach space. If either $X^*$ or $Y^{*}$ has the AP, then $\NA(X^*\pten Y^*)=X^*\pten Y^*$.
\end{thm}

\begin{proof}
    Let $i\colon X\rightarrow Z$ be the isometric embedding and $Id\colon Y\rightarrow Y$ be the identity operator. Since the injective tensor product respect subspaces, $i\otimes Id \colon X\iten Y\rightarrow Z\iten Y$ is a linear isometry \cite[Proposition 3.2]{ryan}. Now, take $z\in (X\iten Y)^* = X^*\pten Y^{*}$ (this identification follows by \cite[Theorem 5.33]{ryan} because $X^*$ has the RNP since $Z^*=\ell_1(I)$ has the RNP). So, by Lemma \ref{lema:operadorcociente} the operator $(i\otimes Id)^*=i^*\otimes Id^* \colon \ell_1(I)\pten Y^{*}\rightarrow X^*\pten Y^{*}$ is a quotient operator and there is some $u\in \ell_1(I)\pten Y^{*}$ with $\|u\|=\|z\|$ and $i^*\otimes Id^* (u)=z$. Since $\NA_\pi(\ell_1(I)\pten Y^{*}) = \ell_1(I)\pten Y^{*}$ (\cite[Proposition~3.6 (b)]{DJRRZ}) we can find sequences $(\lambda_n)_n\subseteq \ell_1(I)$ and $(y^{*}_n)_n\subseteq Y^{*}$ such that
    \begin{equation*}
        u=\sum_{n=1}^{\infty} \lambda_n\otimes y^{*}_n \quad \text{and}\quad \|u\|=\sum_{n=1}^{\infty} \|\lambda_n\|\|y^{*}_n\|.
    \end{equation*}
    Therefore, denoting $x^*_n = i^*(\lambda_n)$ for every $n\in\natu$, we have that
    \begin{equation*}
        z=(i^*\otimes Id^*)(u) = \sum_{n=1}^{\infty} (i^*\otimes Id^*) (\lambda_n\otimes y^{*}_n) = \sum_{n=1}^{\infty} i^*(\lambda_n)\otimes Id^*(y^{*}_n) = \sum_{n=1}^{\infty} x^*_n\otimes y^{*}_n
    \end{equation*}
    and, since $\|i^*\|\leq 1$, we obtain
    \begin{equation*}
        \|z\|\leq \sum_{n=1}^{\infty} \|x^*_n\|\|y^{*}_n\| \leq \sum_{n=1}^{\infty} \|\lambda_n\|\|y^{*}_n\|=\|u\|=\|z\|.
    \end{equation*}
    Thus, $z\in \NA_\pi(X^*\pten Y^*)$.\end{proof}

An application of Lemma~\ref{lema:complemented} and the above theorem yields the following result:

\begin{coro}\label{coro:subspredl1attains}
Let $Z$ be a Banach space such that $Z^*=\ell_1(I)$ isometrically. Let $X$ be a subspace of $Z$ and $Y$ be any Banach space which is $1$-complemented in $Y^{**}$. If either $X^*$ or $Y^{**}$ has the AP, then $\NA(X^*\pten Y)=X^*\pten Y$.
\end{coro}

Before exhibiting new examples of spaces where all the elements attain their projective norm, it is pertinent to explain how Theorem~\ref{teo:subspredl1attains} generalises \cite[Theorem 4.1]{DGLJRZ} (the generalisation of \cite[Proposition 3.6]{DJRRZ} is evident).

\begin{rema}
If $X$ is a finite-dimensional Banach space whose unit ball is a polytope,  then $B_{X^*}$ is a finite intersection of halfspaces and so it is a polytope too.  
Consequently, $B_{X^*}=\co(\{f_1,\ldots, f_n\})$, where $f_1,\ldots, f_n$ are the extreme points of $B_{X^*}$. Now the mapping
$$\begin{array}{ccc}
    X & \longrightarrow & \ell_\infty^n \\
    x & \longmapsto & (f_1(x),\ldots, f_n(x))
\end{array}$$
is a linear into isometry since $\norm{x}=\sup\{|x^*(x)|: x\in B_{X^*}\}=\max\{|f_i(x)|: i\in\{1,\ldots,n\}\}$ for $x\in X$.
\end{rema}

A well-known example of Banach space $Z$ satisfying $Z^*=\ell_1(I)$ is $C(K)$ where $K$ is a scattered compact topological space \cite[Theorem~14.24]{RojoYAmarillo}. Hence the following result holds.

\begin{coro}\label{cor:C(K)attains} Let $X$ be a subspace of $C(K)$, where $K$ is a compact Hausdorff scattered space, and $Y$ be any Banach space which is $1$-complemented in $Y^{**}$. If $X^*$ has the AP then $\NA(X^*\pten Y)=X^*\pten Y$.
\end{coro}

 The following example shows that the hypothesis of $K$ being scattered cannot be removed in Corollary~\ref{cor:C(K)attains}. Moreover, note that $C(\mathbb T)^*\pten \mathbb R^2$ is a dual space, therefore the $w^*$-compactness of the ball is not enough to ensure that every tensor attains its norm.

\begin{ejem}
    We are going to see that 
    \begin{equation*}
        \NA\left(C(\T)^*\ptensor \R^2\right) \neq C(\T)^*\ptensor \R^2,
    \end{equation*}
    where $\mathbb T$ stands for the unit circle of the euclidean space $\mathbb R^2$.     Firstly, endow $\T$ with the Haar measure and define $\phi \colon L_1(\T) \rightarrow C(\T)^*$, such that for $f\in L_1(\T)$ the functional $\phi(f)$ is given by 
    \begin{equation*}
        \phi(f)(g) = \int_{\T} g(t)f(t)\text{ d}t, \quad \text{for every } g\in C(\T).
    \end{equation*}
    It is well known that $\phi$ is a linear into isometry whose image is precisely
    the space of absolutely continuous measures with respect to the Haar measure. It follows from Lebesgue's decomposition that there exists a projection $P\colon C(\T)^*\rightarrow L_1(\T)$ with $\|P\|=1$. 
    Now, Example 3.12(a) in \cite{DJRRZ} shows that there exists an element in $L_1(\T)\pten \mathbb R^2$ which does not attain its projective norm. By Lemma \ref{lema:complemented}, the same holds for $C(\T)^*\pten \mathbb R^2$. 
\end{ejem}

In what follows we will show some new examples where Corollary~\ref{cor:C(K)attains} applies. To this end, we will see that certain spaces of Lipschitz functions embed into $C(K)$ for a scattered $K$. 

\begin{prop}{\label{prop:S0(M)}}
    Let $M$ be a countable proper metric space. Then, there exists a countable compact set $K$ and a linear into isometry $\phi \colon \lip(M) \rightarrow C(K)$.
\end{prop}

\begin{proof}
    Let $K=(M^2\setminus \Delta) \cup \{\infty\}$ be Alexandroff compactification of $M^2\setminus \Delta$, which is indeed a countable compact set.
    We define $\phi \colon \lip(M) \rightarrow C(K)$  
    as follows:
     \begin{equation*}
         \phi(f)(z) = \left\{\begin{array}{ccc}
             \frac{f(x)-f(y)}{d(x,y)} & \text{if} & z=(x,y) \in M^2\setminus \Delta \\
              0& \text{if} &  z=\infty. 
         \end{array}\right.
     \end{equation*}
     Since $f$ is flat at infinity,  for every $\varepsilon>0$, we can find $r >0$ such that $\frac{|f(x)-f(y)|}{d(x,y)} < \varepsilon$ if $x\notin B(0,r)$ or $y\notin B(0,r)$. Furthermore, since $M$ is locally flat, it is uniformly locally flat on $B(0,r)$, thus there exists $0<\delta<r$ such that $\frac{|f(x)-f(y)|}{d(x,y}<\varepsilon$ whenever $0<d(x,y)<\delta$. Then, the set
     \[L=\{(x,y)\in M^2\setminus \Delta : x,y\in B(0,r), d(x,y)\geq \delta\}\]
     is compact, so its complement $K\setminus L$ is an open neighbourhood of $\infty$ with  $|\phi(f)(z)|<\varepsilon$ for every $z\in K\setminus L$. This shows that $\phi(f)\in C(K)$ so $\phi$ is well defined, and clearly it is also linear. Finally, $\phi$ is an into isometry since 
     \begin{align*}
         \|\phi(f)\|_{\infty} &= \sup \{|\phi(f)(z)|\;:\; z\in (M^2\setminus \Delta) \cup \{\infty\} \} \\&= \sup \left\{\frac{|f(x)-f(y)|}{d(x,y)}\;:\; (x,y)\in M^2\setminus \Delta \right\} = \|f\|_L
     \end{align*}
     holds, and the proof is finished.
\end{proof}

\begin{prop}\label{prop:liptau(M)}
    Let $(M,d)$ be a countable metric space, and $\tau$ be some topology such that $(M,\tau)$ is compact. Suppose that $d$ is $\tau$-lower semicontinuous and $V=\{d(x,y) \;:\; (x,y)\in M^2\}\subseteq \R_0^+$ is a compact set. Then, there exists a  countable compact set $K$ and a linear into isometry $\phi \colon \lipC(M) \rightarrow C(K)$.
\end{prop}

\begin{proof}
    The proof is a slight modification of the one of Theorem~6.2 in \cite{Kalton04}. Notice first that $V$ is countable, and $(M,\tau)\times (M,\tau)\times V$ is a countable compact space. 
    Furthermore, let $(x_n,y_n,t_n)_n$ be a sequence on $(M,\tau)\times (M,\tau)\times V$ converging to $(x,y,t)$, and satisfying $d(x_n,y_n)\leq t_n$ for every $n\in\natu$. Since $d$ is $\tau$-l.s.c., we have
    \begin{equation*}
        d(x,y) \leq \liminf_{n\to\infty} d(x_n,y_n) \leq \liminf_{n\to\infty} t_n = t,
    \end{equation*}
    so $(x,y,t)$ is also in \[K:=\{(x,y,t):(M,\tau)\times (M,\tau)\times V : d(x,y)\leq t\}.\]
    Thus, $K$ is a closed set within $(M,\tau)\times (M,\tau)\times V$, which means $K$ is compact and countable. 

    Now, we define $\phi \colon \lipC(M) \rightarrow C(K)$ by 
    \begin{equation*}
        \phi(f)(x,y,t) = \left\{\begin{array}{ccc}
             \frac{f(x)-f(y)}{t} & \text{if} & t\neq 0\\
              0& \text{if} &  t=0. 
         \end{array}\right.
    \end{equation*}
    Firstly, we have
    \begin{equation*}
        \frac{|f(x)-f(y)|}{t} \leq \frac{|f(x)-f(y)|}{d(x,y)} \underset{t\to 0}{\longrightarrow} 0,
    \end{equation*}
    where the limit is due to the fact that $f\in\lipC(M)$. Thus, $\phi(f)$ is continuous and $\phi$ is well defined. Moreover, $\phi$ is clearly linear and
    \begin{align*}
        \|\phi(f)\|_{\infty} &= \sup \left\{|\phi(f)(x,y,t)| \;:\; (x,y,t)\in K\right\} \\&=  \sup \left\{\frac{|f(x)-f(y)|}{t} \;:\; (x,y,t)\in K, \; t\neq 0\right\} \\&=  \sup \left\{\frac{|f(x)-f(y)|}{d(x,y)} \;:\; (x,y)\in M^2\setminus \Delta \right\} = \|f\|_L,
    \end{align*}
    since the third equality follows since $t\geq d(x,y)$ for $(x,y,t)\in K$, and $(x,y,d(x,y))\in K$. This proves that $\phi$ is an isometry.
\end{proof}

A combination of Theorem~\ref{thm:dualityfree} together with Propositions~\ref{prop:S0(M)} and \ref{prop:liptau(M)} yield the following result on projective norm-attainment.
\begin{coro}\label{coro:lipschitfreetodotensor}
    Let $M$ be a complete metric space. Then $\NA(\mathcal F(M)\pten Y)=\mathcal F(M)\pten Y$ for any Banach space $Y$ which is $1$-complemented in $Y^{**}$ if $M$ satisfies one of the following conditions:
    \begin{enumerate}
        \item[a)] $M$ is countable and proper.
        \item[b)] $M$ is uniformly discrete, countable, and there is a compact Hausdorff topology $\tau$ on $M$ such that $d$ is $\tau$-lower semicontinuous, and $V=\{d(x,y) \;:\; (x,y)\in M^2\}\subseteq \R_0^+$ is a compact set.
    \end{enumerate}
 
\end{coro}

\begin{proof}
In a), we have the duality $\mathcal F(M)=\lip(M)^*$, and $\mathcal F(M)$ has the approximation property \cite[Theorem 2.6]{Dalet05}, whereas in b) we have $\mathcal F(M)=\lipC(M)^*$ and $\mathcal F(M)$ has the approximation property \cite[Theorem 4.4]{Kalton04}. Now, the result follows from Propositions~\ref{prop:S0(M)} and \ref{prop:liptau(M)}, Corollary~\ref{cor:C(K)attains}, and the fact that any compact countable space is scattered. 
\end{proof}

Now let us show an example where the above Corollary applies.

\begin{ejem}\label{ex:freespaceexamppropio}
There are metric spaces $M$ for which b) in Corollary~\ref{coro:lipschitfreetodotensor} applies and which are not isometrically isomorphic to $\ell_1$. Indeed, define $M:=\{0\}\cup\{x_n:n\in\mathbb N\}\cup \{z\}$ whose distances are defined as $d(x_n,x_m)=d(x_n,0)=d(x_n,z)=1$ holds for every $n,m\in\mathbb N, n\neq m$ and $d(0,z)=2$.

First we claim that $\mathcal F(M)$ is not isometrically isomorphic to $\ell_1$. Indeed, an application of \cite[Theorem 4.3]{pr18} implies that the element $\frac{\delta(z)}{2}\in S_{\mathcal F(M)}$ is a point of Fr\'echet differentiability of the norm of $\mathcal F(M)$. However, it is a well-known fact that the norm of $\ell_1$ is nowhere Fr\'echet differentiable (c.f. e.g. \cite[Examples I.1.6, c)]{DGZ}). Consequently, $\ell_1$ and $\mathcal F(M)$ cannot be isometrically isomorphic.

On the other hand, let us prove that there exists a 
topology on $M$ which satisfies the assumptions of b) in Corollary~\ref{coro:lipschitfreetodotensor}. In order to do so, we declare all the points of $M$ to be discrete points with the exception of the point $x_1$. For $x_1$, we define the following the following family of sets
$$U_n:=\{x_k, k\geq n\}$$
to be a neighbourhood basis for the point $x_1$. It is immediate that the above defines a compact topology on $M$, and  
that the set $\{d(x,y): (x,y)\in M^2\}=\{0,1,2\}$ is a compact set.

Finally, let us prove that $d\colon (M,\tau)^2\longrightarrow \mathbb R$ is lower semicontinuous. In order to do so, select a convergent net $\{(x_s,y_s)\}\rightarrow (x,y)$ in $(M,\tau)^2$, and let us prove that $d(x,y)\leq \liminf d(x_s,y_s)$. Clearly we may assume $x\neq y$, so $x_s\neq y_s$ eventually and thus $\liminf d(x_s,y_s)\geq 1$. Therefore, the only case where the inequality may not hold is when $d(x,y)=2$, that is, $\{x,y\}=\{0,z\}$. But $0$ and $z$ are isolated in $(M,\tau)$, so in that case $\{x_s,y_s\}=\{0,z\}$ eventually. Thus $d(x,y)=\liminf d(x_s,y_s)$ and we are done. 
\end{ejem}

As far as we are concerned, the only known example of Banach space satisfying $\NA_\pi(X\pten Y)=X\pten Y$ for every Banach space $Y$ is $X=\ell_1(I)$ \cite[Proposition 3.6]{DJRRZ}. In the following result we investigate the above condition in the search for a necessary condition. Keeping in mind that, if a Banach space $X$ satisfies that $\NA_\pi(X\pten Y)=X\pten Y$  for every Banach space $Y$ then $\NA(X,Y^*)$ is dense in $\mathcal L(X,Y^*)$, we conclude that $X$ must enjoy a kind of Lindenstrauss property A for operators taking value on dual spaces. Because of that, we will make use of the ideas of \cite[Theorem 2]{lindenstrauss} to obtain the following result.

\begin{thm}\label{teo:idealindens} Let $X$ be a separable Banach space. Assume that  $\NA(X^*, Y^*)$ is dense in $\mathcal L(X^*, Y^*)$ for every Banach space $Y$. Then $B_{X^*}=\cconv(\exp B_{X^*})$.
\end{thm}

\begin{proof}
    Suppose that $B_{X^*}\neq\cconv(\exp B_{X^*})$. By Hahn-Banach theorem and observing that $\exp B_{X^*}$ is symmetric, we can find some $\varphi\in S_{X^{**}}$ and $\delta >0$ such that 
    \begin{equation*}
        |\varphi(x^*)|\leq 1-\delta,\quad \forall x^*\in \exp B_{X^*}.
    \end{equation*}
    Since $X$ is separable, it admits a renorming $\renorm{\cdot}$ such that $\renorm{x^*}\leq \norm{x^*}$ for every $x^*\in X^*$ and $(X^*,\renorm{\cdot})$ is strictly convex (\cite[Theorem II.2.6.(i)]{DGZ}). Defining $Y=X\oplus_2 \R$, it follows that $Y^*= X^*\oplus_2\R$. Given $M>\frac{2}{\delta}$, we consider the operator $V\colon X^*\rightarrow Y^*$ with $V(x^*)\coloneqq (x^*, M\varphi(x^*))$ for each $x^*\in X^*$. Then, $V$ is an isomorphism (into) and
    \begin{equation*}
        \norm{V}\geq M, \quad \norm{V(x^*)}\leq M-1 \text{ for each } x^*\in \exp B_{X^*}.
    \end{equation*}
    It is easy to see that if $G\in \mathcal L(X^*, Y^*)$ and $\norm{G-V}\leq \eta$, for $\eta=\frac{1}{2\norm{V^{-1}}}$, then $G$ is also an isomorphism (into). Therefore, since  $\NA(X^*, Y^*)$ is dense in $\mathcal L(X^*, Y^*)$, we can take $G\in \mathcal L(X^*, Y^*)$ with $\norm{G-V}<\min\left\{\frac{1}{2},\eta\right\}$ and such that $G(x_0^*)=\norm{G}$ for some $x_0^*\in S_{X^*}$.

    Now, we claim that the functional $x_0^*$ is exposed in $B_{X^*}$. Indeed, by Hahn-Banach theorem, there is $y^*\in S_{Y^{**}}$ such that $y^*(G(x_0^*))= \norm{G}$. Then, we define $$z \coloneqq \frac{y^*\circ G}{\norm{y^*\circ G}}\in S_{X^{**}}, $$
        and we are going to prove that $z$ exposes $x_0^*$. Suppose that there is another $v^*\in S_{X^*}$ such that $z(v^*)=1$. Therefore, $y^*(G(v^*))=\norm{G}$ and 
        \begin{equation*}
            \norm{G(x_0^*)+G(v^*)} \geq y^*(G(x_0^*))+ y^*(G(v^*)) = \norm{G}+\norm{G} = \norm{G(x_0^*)}+\norm{G(v^*)}.
        \end{equation*}
        Noticing that $Y^*$ is strictly convex, we conclude $G(x_0^*) = G(v^*)$. Finally, $x_0^*=v^*$ because $G$ is injective. This proves the claim, that is, $x_0\in \exp B_{X^*}$.

    Finally, we have that
    \begin{align*}
        M-1&\geq \norm{V(x_0^*)} \geq \norm{G(x_0^*)}-\norm{G-V} = \norm{G}-\norm{G-V} \\&\geq \norm{V}-2\norm{G-V} > M-1,
    \end{align*}
    which is a contradiction. Thus, it must be $B_{X^*}=\cconv(\exp B_{X^*})$.
\end{proof}

As we have mentioned above, if a Banach space $X$ satisfies that\\ $\NA_\pi(X\pten Y)=X\pten Y$ then $\NA(X,Y^*)$ must be dense in $\mathcal L(X,Y^*)$. Hence, the following is a direct consequence of Theorem~\ref{teo:idealindens}.

\begin{coro}\label{cor:propAnormatensor} Let $X$ be a separable Banach space. Assume that\\ $\NA(X^*\pten Y)=X^*\pten Y$ for any Banach space $Y$. Then $B_{X^*}=\cconv(\exp B_{X^*})$.
\end{coro}

In \cite[Corollary 3.11]{DJRRZ} it is proved that a necessary condition to get that $X\pten Y\setminus \NA_\pi(X\pten Y)\neq\emptyset$ is to require that $\NA(X,Y^*)$ is not dense in $\mathcal L(X,Y^*)$. We will end the section by showing that the connection between the size of the sets $X\pten Y\setminus \NA_\pi(X\pten Y)$ and $\mathcal L(X,Y^*)\setminus\NA(X,Y^*)$ goes beyond and, if we require the non-density of $\operatorname{span}(\NA(X,Y^*))$, we get that $X\pten Y\setminus \NA_\pi(X\pten Y)$ not only is non-empty but also it is large enough to generate the unit ball of $B_{X\pten Y}$ by taking closed convex hull.

Let us start with a result of theoretical nature which establishes a sufficient condition in order to $X\pten Y\setminus\NA_\pi(X\pten Y)$ be norming for $\mathcal L(X,Y^*)$.

\begin{thm}\label{theo:complebig}
Let $X, Y$ be Banach spaces satisfying that the set $Z:=$\\$\NA(X,Y^*)-\NA(X,Y^*)$ is not dense in $\mathcal L(X,Y^*)$. 
Then,
$$B_{X\pten Y}=\overline{\co}(\{z\in X\pten Y\setminus\NA_\pi(X\pten Y): \Vert z\Vert\leq 1\}).$$
\end{thm}

\begin{proof} We will show first that $\NA(X, Y^*)$ is nowhere dense in $\mathcal L(X,Y^*)$. To this end, let $G\in \overline{\NA(X,Y^*)}$ and $\varepsilon>0$, and take $U\in \NA(X,Y^*)$ with 
$$\Vert G-U\Vert<\frac{\varepsilon}{2}.$$
Now let $T\in \mathcal L(X,Y^*)$ be such that $d(T,Z)>0$ and $\norm{T}=1$, and set $V:=U+\frac{\varepsilon}{2} T$. We have $\Vert G-V\Vert<\varepsilon$. Also, given any $W\in \NA(X,Y^*)$, we have that $U,W\in \NA(X,Y^*)$ and we infer that
$$\norm{V-W}=\left\Vert (U-W)+\frac{\varepsilon}{2}T\right\Vert=\frac{\varepsilon}{2}\left\Vert\frac{W-U}{\varepsilon/2}-T\right\Vert\geq\frac{\varepsilon}{2}d(T,Z)>0$$
since $\frac{U-W}{\varepsilon/2}\in Z$ since the set $Z$ is clearly closed under scalar multiplication operator. Consequently we have that $V$ is an operator with $\Vert G-V\Vert<\varepsilon$ and $d(V,\NA(X,Y^*))>0$. This shows that $\overline{\NA(X,Y^*)}\subset \overline{\mathcal L(X, Y^*)\setminus \overline{\NA(X,Y^*)}}$, so indeed we have
\[ \mathcal L(X, Y^*)\subset \overline{\mathcal L(X, Y^*)\setminus \overline{\NA(X,Y^*)}}.\]

Now, write $A:=\{z\in X\pten Y\setminus\NA_\pi(X\pten Y): \Vert z\Vert\leq 1\}$. In order to prove that $B_{X\pten Y}:=\overline{\co}(A)$ we need to check, by Hahn-Banach theorem, that 
$$\Vert G\Vert=\sup_{z\in A} G(z) \quad \forall G\in \mathcal L(X,Y^*).$$
To this end, let $G\in \mathcal L(X,Y^*)$ with $\Vert G\Vert=1$ and $\varepsilon>0$, and let us find $z\in A$ with $G(z)>1-\varepsilon$. 
By the first part of the proof, there is $V\colon X\to Y^*$ such that $\delta:=d(V, \NA(X, Y^*))>0$ and $\norm{G-V}<\frac{\varepsilon}{4}$. By Bishop-Phelps theorem, we can find $S\in  \mathcal L(X,Y^*)$ and $z\in S_{X\pten Y}$ with $S(z)=\Vert S\Vert$, such that $\Vert V-S\Vert<\min\{\frac{\delta}{2}, \frac{\varepsilon}{4}\}$. Thus,  $d(S,\NA(X,Y^*))\geq \delta/2 >0$. Since $S$ does not attain its norm as operator we derive that $z\notin \NA_\pi(X\pten Y)$. On the other hand, we get
\begin{align*}
    G(z)&\geq S(z)-\Vert G-S\Vert  =\Vert S\Vert-\Vert G-S\Vert \\ 
& \geq \Vert G\Vert-2\Vert G-S\Vert
 =1-2(\Vert G-V\Vert+\norm{V-S})\\
& >1-\varepsilon,
\end{align*}
which completes the proof.
\end{proof}

Before providing examples where the above result applies, a remark is pertinent about the assumptions made there.

\begin{rema}\label{rem:na-nadenso}
The denseness of $\NA(X,Y^*)-\NA(X,Y^*)$ may happen frequently. Indeed, it is known that if $\NA(X,Y^*)$ is a residual set (i.e. if it contains a $G_\delta$ dense set) then $\mathcal L(X,Y^*)=\NA(X,Y^*)-\NA(X,Y^*)$ (c.f. e.g. \cite[Proposition 1.7]{jmr23}). A big class of examples of Banach spaces $X$ and $Y$ for which $\NA(X,Y^*)$ is residual is exhibited in \cite{jmr23}. This turns out to happen, for instance, if $X$ has the RNP or property $\alpha$ and $Y$ is any Banach space \cite[Proposition 1.1 and Lemma 1.2]{jmr23}.

This shows in particular that, in order to apply Theorem~\ref{theo:complebig}, we must first avoid such kind of hypotheses.
\end{rema}

Now we can find examples where the above theorem applies.

\begin{ejem}\label{exam:fovelle}
Let $Y$ be a Banach space such that $Y^*$ is asymptotically midpoint uniformly convex (AMUC) and satisfies one of the following conditions:
\begin{itemize}
    \item $Y^*$ has a normalized, symmetric basic sequence which is not equivalent to the unit vector basis in $\ell_1$,
    \item $Y^*$ has a normalized sequence with upper $p$-estimates for some $1<p<\infty$.
\end{itemize}
In the proof of \cite[Theorem A]{fovelle} it is proved that there exists a Banach space $X$ (indeed, a predual of a Lorentz sequence space $d_*(w)$) with a Schauder basis $\{e_n\}$ with the following properties:
\begin{enumerate}
    \item $\NA(X,Y^*)\subseteq V:=\{G\in \mathcal L(X,Y^*): \lim_n \Vert G(e_n)\Vert=0\}$ \cite[Proposition 3]{fovelle}. In particular, $V$ contains $\NA(X,Y^*)-\NA(X,Y^*)$ since $V$ is a vector space.
    \item There exists $T\in S_{\mathcal L(X,Y^*)}$ such that $\Vert T(e_n)\Vert=1$ holds for every $n\in\mathbb N$ \cite[Propositions~4 and 6]{fovelle}. In particular $d(T,V)=1$.
\end{enumerate}
In particular, under the conditions exposed above, Theorem~\ref{theo:complebig} applies.

We thank Miguel Mart\'in for letting us know about the existence of the paper \cite{fovelle}.
\end{ejem}

\begin{rema}
In \cite[Example 3.12, c)]{DJRRZ} it is proved that there exists a Banach space $G$ for which there are elements in $G\pten \ell_p$, $1<p<\infty$ which do not attain their projective norm. Observe that Example~\ref{exam:fovelle} applies for $Y=\ell_p$, from where we deduce that, for the space $X$ described in Example~\ref{exam:fovelle}, the set $X\pten Y\setminus\NA_\pi(X\pten Y)$ is norming for the dual.

Observe that, since $X$ has a monotone Schauder basis, both $X$ and $Y$ have the metric $\pi$-property, so in particular $\NA_\pi(X\pten Y)$ is dense in $X\pten Y$.
\end{rema}

\begin{rema}\label{remark:bignaandcomple}
In \cite[Section 3]{rueda23} it is proved that if $X$ is a Banach space with the metric $\pi$-property whose norm depends upon finitely many coordinates and $Y$ is a Hilbert space, then both $\NA_\pi(X\pten Y)$ and $X\pten Y\setminus\NA_\pi(X\pten Y)$ are dense sets in $X\pten Y$.

As far as we are concerned, the above is the first (and the single) example of projective tensor product where $\NA_\pi$ and its complement are both big sets.

Observe that Example~\ref{exam:fovelle} enlarges the class of examples where such behaviour happens. Indeed, the Banach space $X$ described there not only has the metric $\pi$-property but also a monotone Schauder basis. Hence, if we take $Y$ as in Example~\ref{exam:fovelle} with the metric $\pi$-property, we get that $\NA_\pi(X\pten Y)$ is dense in $X\pten Y$ and $B_{X\pten Y}=\overline{\co}(\{z\in S_{X\pten Y}: z\notin \NA_\pi(X\pten Y)\})$.
\end{rema}

\section{Denseness of norm attaining tensors}\label{sect:densidad}

In this section we aim to obtain new results about density of the set of tensors that attain its projective norm. The improvement will be done in two different directions: on the one hand, Theorem~\ref{th:general} below will establish density on $X\pten Y$ making assumptions only on one of the factors (say $X$); on the other hand, from an application of Lemma~\ref{lema:complementeddensity}, which will allow us to replace assumptions of $Y$ being a dual space with the assumption that $Y$ is $1$-complemented in its bidual.

Note first that the proof of Lemma \ref{lema:complemented} also yields  that if $\NA_\pi(X\pten Y)$ is dense in $X\pten Y$ then $\NA_\pi(X\pten Z)$ is dense in $X\pten Z$. Indeed, if we consider $z\in X\pten Z\subseteq X\pten Y$ and $\varepsilon>0$ then, using the density of $\NA_\pi(X\pten Y)$ we can find $z'\in \NA_\pi(X\pten Y)$ with $\Vert z-z'\Vert_{X\pten Y}<\varepsilon$. As in the proof of Lemma~\ref{lema:complemented} we can prove that $(Id\otimes P)(z')\in \NA_\pi(X\pten Z)$. On the other hand, since $z\in X\pten Z$ we derive $z=(Id\otimes P)(z)$, so
$$\Vert z-(Id\otimes P)(z')\Vert_{X\pten Z}=\Vert (Id\otimes P)(z-z')\Vert_{X\pten Z}\leq \Vert Id\otimes P\Vert \Vert z-z'\Vert_{X\pten Y}<\varepsilon.$$
Consequently, the following lemma also does hold.

\begin{lema}\label{lema:complementeddensity} Let $X,Y$ be Banach spaces such that $\overline{\NA_\pi(X\pten Y)}=X\pten Y$. Assume $Z\subset Y$ is a $1$-complemented subspace.  Then, $\overline{\NA_\pi(X\pten Z)}=X\pten Z$.
\end{lema}

Now, we will prove the following abstract result from where we will obtain all our improvements.

\begin{thm}\label{th:general}
    Let $X,Y$ be Banach spaces. Suppose that for every $\varepsilon>0$, $n\in\natu$ and $x_1,\ldots,x_n \in X$, there exists a finite dimensional subspace $Z\subseteq X$ which is 1-complemented in $X$ and such that we can find $x_i'\in Z$ with $\|x_i-x_i'\|<\varepsilon$ for each $i=1,\ldots,n$. Assume that $\overline{\NA_\pi(Z\ptensor Y)}=Z\ptensor Y$, for the previous $Z$. Then,
    \begin{equation*}
        \overline{\NA_\pi(X\ptensor Y)}=X\ptensor Y.
    \end{equation*}
\end{thm}

\begin{proof}
    Take $z= \sum_{i=1}^n x_i\otimes y_i \in X\ptensor Y$, and $\varepsilon>0$. We may assume $\norm{y_i}=1$ for each $i$. By hypothesis, we can find $Z\subseteq X$ as in the statement, for $x_1,\ldots,x_n$ and $\varepsilon' = \frac{\varepsilon}{2n}$. Recall that, since $Z$ is 1-complemented, we have $\norm{u}_{X\pten Y}=\norm{u}_{Z\pten Y}$ for each $u\in Z\pten Y$. 
    \\We define $z'=\sum_{i=1}^n x_i'\otimes y_i$, which satisfies $\|z-z'\|\leq \sum_{i=1}^n \norm{x_i-x_i'}\norm{y_i} < \frac{\varepsilon}{2}$. Since $z'\in Z\ptensor Y = \overline{\NA_\pi(Z\ptensor Y)}$, there exists $z''=  \sum_{k=1}^{\infty} u_k\otimes v_k \in  Z\ptensor Y$ such that $\|z''\|= \sum_{k=1}^{\infty} \|u_k\|\|v_k\|$ and $\|z'-z''\|<\frac{\varepsilon}{2}$. Furthermore, since $Z$ is 1-complemented in $X$, we have that
    \begin{equation*}
        \|z''\|_{X\ptensor Y} = \|z''\|_{Z\ptensor Y} = \sum_{k=1}^{\infty} \|u_k\|\|v_k\|,
    \end{equation*}
    so $z''\in \NA_\pi(X\ptensor Y)$. Moreover,
    \begin{align*}
        \|z-z''\|_{X\ptensor Y}&\leq \|z-z'\|_{X\ptensor Y}+\|z'-z''\|_{X\ptensor Y}\\
        &=\|z-z'\|_{X\ptensor Y}+\|z'-z''\|_{Z\ptensor Y} < \frac{\varepsilon}{2}+\frac{\varepsilon}{2} = \varepsilon.
    \end{align*}
    Thus, the density of $\NA_\pi(X\ptensor Y)$ is proved.
\end{proof}

Now we will obtain some consequences of the above theorem. 

\begin{coro}\label{cor:denseness} $\NA_\pi(X\pten Y)$ is dense in $X\pten Y$ in the following cases:
\begin{enumerate}[a)]
\item $X=L_1(\mu)$ for some measure $\mu$, and $Y$ is any Banach space.
\item $X^*=L_1(\mu)$ and $Y$ is $1$-complemented in $Y^{**}$.
\item $X$ has the metric $\pi$-property and $Y$ is a dual Banach space with the RNP.
\item $X$ has the metric $\pi$-property, $Y$ is $1$-complemented in $Y^{**}$ and $Y^{**}$ has the RNP.
\end{enumerate}
\end{coro}

\begin{proof}
a) Given $f_1,\ldots, f_n\in L_1(\mu)$ and $\varepsilon>0$, we can find disjoint measurable sets $A_1, \ldots, A_m$ and numbers $a_{ij}\in \mathbb R$ such that the simple functions $f_i'=\sum_{j=1}^m a_{ij} \chi_{A_{j}}$ satisfy $\norm{f_i-f_i'}<\varepsilon$ for each $i$, and then the space $Z=\operatorname{span}\{\chi_{A_j}: 1\leq j\leq m\}$ is $1$-complemented in $L_1(\mu)$ (see Example 4.12.b) in \cite{DJRRZ}). Notice that $Z=\ell_1^m$ isometrically. Thus $\NA_\pi(Z\pten Y)=Z\pten Y$ by \cite[Proposition~3.6]{DJRRZ}. Now, the result follows from Theorem \ref{th:general}. 

b) Given $x_1,\ldots, x_n\in X$ and $\varepsilon>0$, we can find a finite-dimensional subspace $Z\subset X$ with $Z=\ell_\infty^n$ isometrically (and hence $1$-complemented in $X$) and vectors $x_i'\in Z$ with $\norm{x_i-x_i'}<\varepsilon$ \cite[Example 4.12.c)]{DJRRZ}. Then $B_Z$ is a polytope and so $\NA_\pi(Z\pten Y)=Z\pten Y$ \cite[Theorem 4.1]{DGLJRZ}, so Theorem~\ref{th:general} applies.    

c) Corollary 4.6 in \cite{DGLJRZ} shows that $\overline{\NA_\pi(Z\pten Y)}=Z\pten Y$ for any finite-dimensional space $Z$. Thus we may apply Theorem \ref{th:general}. 

d) Follows by a combination of c) and Lemma~\ref{lema:complementeddensity}.
\end{proof}

\begin{rema}
Observe that the assumptions on $Y$ considered in c)  of Corollary~\ref{cor:denseness} do not imply the ones in d). For instance, taking $Y=\ell_1$, it follows that $Y$ is a dual Banach space with the RNP  
but its bidual $\ell_1^{**}=\ell_\infty^*$ fails the RNP, for instance, since $\ell_\infty$ contains an isometric copy of $\ell_1$.

On the other hand, we have not found any explicit example of Banach space $X$ such that $X^{**}$ has the RNP, that there exists a bounded linear projection $P\colon X^{**}\longrightarrow X$ with $\Vert P\Vert\leq 1$ and such that $X$ fails to be isometrically a dual Banach space.

Observe that such example, in case it exists, should satisfy that both the construction of the space and the proof of the requirements must be involved, essentially because the space $X$ should satisfy the usual properties whose failure are the common test to prove that a given Banach space is not a dual space (i.e. $1$-complementability in the bidual and the well known result that separable dual spaces have RNP).

Indeed, observe that if we remove the condition that $\Vert P\Vert\leq 1$ then such example does exist: in \cite[Theorem 1.1]{broito} it is proved that there exists an equivalent norm $|||\cdot|||$ on the predual of the James space $J$,  
such that $(J,|||\cdot|||)$ fails to be isometric to any dual Banach space, whereas the complementability in its bidual and the RNP on its bidual survive because both conditions are stable by taking equivalent renormings. The authors thank Gin\'es L\'opez-P\'erez for pointing them this result.

The above makes us think that such space $X$ should exist and, consequently, the assumptions made on d) should be independent of that of c).

\end{rema}

\section*{Acknowledgements}  

The authors are grateful to Audrey Fovelle, Gin\'es L\'opez-P\'erez and Miguel Mart\'in for useful discussions on the topic of the paper.

Part of the results were obtained during the visit of the last author to Universidad de Zaragoza in April 2024. He would like to express his gratitude to Departamento de An\'alisis Matem\'atico of UNIZAR and to IUMA for the hospitality received during the stay.

The research of Luis C. García-Lirola was supported by \\
MCIN/AEI/10.13039/501100011033: grants PID2021-122126NB-C31 and \\PID2022-137294NB-I00, by DGA: grant E48-20R and by Generalitat Valenciana: grant CIGE/2022/97.

The research of Abraham Rueda Zoca was supported by \\
MCIN/AEI/10.13039/501100011033: grant PID2021-122126NB-C31, Junta de Andaluc\'ia: grant FQM-0185, by Fundaci\'on S\'eneca: ACyT Regi\'on de Murcia: grant 21955/PI/22 and by Generalitat Valenciana: grant \\ CIGE/2022/97.

\end{document}